\documentclass{article}
\usepackage{graphicx}
\usepackage[left=3cm, right=3cm, top=4cm, bottom=4cm]{geometry}
\usepackage[utf8]{inputenc}
\usepackage{color}
\usepackage{caption}
\usepackage{subcaption}
\usepackage{array}
\usepackage[hidelinks]{hyperref}
\usepackage{graphicx}
\usepackage[square,sort,comma,numbers]{natbib}
\usepackage[utf8]{inputenc}
\usepackage[english]{babel}
\usepackage{amsmath}
\usepackage{mathtools}
\usepackage{amssymb}
\usepackage{textcomp}
\usepackage{blindtext}
\usepackage{amssymb}
\usepackage{pifont}
\usepackage{amsthm}
\usepackage{dsfont}
\usepackage[shortlabels]{enumitem}
\usepackage{subcaption}
\usepackage{float}
\usepackage{tikz-cd}
\usepackage{amsfonts}
\usepackage{mathabx}

\newtheorem{theorem}{Theorem}[section]
\newtheorem{cor}{Corollary}[theorem]

\newtheorem{conj}[theorem]{Conjecture}

\theoremstyle{definition}
\newtheorem{example}[theorem]{Example}
\newtheorem{remark}[theorem]{Remark}

\definecolor{palestinianGreen}{RGB}{0,128,0}
\definecolor{palestinianRed}{RGB}{206,17,38}

\AtEndDocument{\bigskip{\footnotesize
   \textsc{Alejandro Vicente, The Hebrew University of Jerusalem, Jerusalem, Israel.} \par  
  \textit{E-mail}:  \texttt{ kvicente931207@gmail.com}
  }}

\title{The strong Viterbo conjecture and various flavours of duality in Lagrangian products}
\author{Alejandro Vicente}
\date{}

\begin{document}

\maketitle

\begin{abstract}
In this note we analyze normalized symplectic capacities for two different notions of duality in Lagrangian products. Let $\Phi$ be a $n$-tuple of Young functions with Legendre transform $n$-tuple $\Phi^*$ and $K_{\Phi}$ the unit ball for the Luxemburg metric induced by $\Phi$. We can consider the ``dual functional" Lagrangian product $K_{\Phi}\times_LK_{\Phi^*}$ and the usual polar dual Lagrangian product $K_{\Phi}\times_L K_{\Phi}^{\circ}$. We show that for the former, all normalized symplectic capacities agree, while for the latter, we give a lower bound depending on $\Phi$. In particular, under certain conditions on the $n$-tuple $\Phi$, we get that $c(K_{\Phi}\times_L K_{\Phi}^{\circ})=4$, for any normalized symplectic capacity, that is, the strong Viterbo conjecture holds.
\end{abstract}

\section{Introduction}

One of the foundational, and most groundbreaking results in the area of Symplectic Topology, is Gromov's Non-Squeezing Theorem in \cite{Gromov1985PseudoHC}. Let 
\begin{equation*}
    \begin{split}
        B^{2n}(r)&:=\{(x_1,\ldots,x_n,y_1,\ldots,y_n)\in \mathbb{R}^{2n}\mid x_1^2+\ldots x_n^2+y_1^2+\ldots y_n^2\leq r/\pi\},\\
        Z^{2n}(r)&:=B^2(r)\times \mathbb{R}^{2n-2}.
    \end{split}
    \end{equation*}
Gromov's result asserts that, $B^{2n}(r)$ can be symplectically embedded into $Z^{2n}(R)$ if, and only if, $r\leq R$. Beyond the clear implication that symplectic embedding maps are more than just area preserving embeddings, there is more than meets the eye in this seminal result. This is the first instance of what later became known as symplectic capacities.

A \textit{symplectic capacity} $c$ is a map that assigns a non-negative real number or infinite, to a symplectic manifold $(M,\omega)$ such that the following conditions are met:
\begin{itemize}
\item (\textit{Monotonicity}) If $(M_1,\omega_1)$ can be symplectically embedded into $(M_2,\omega_2)$, then $c(M_1,\omega_1)\leq c(M_2,\omega_2)$,
\item (\textit{Conformality}) For $\lambda\neq 0$, we have that $c(M,\lambda \omega)=|\lambda|c(M,\omega)$,
\item (\textit{Non-triviality}) $c(B^{2n}(1),\omega_0)=c(Z^{2n}(1),\omega_0)=1$.
\end{itemize}
In addition, we say that a symplectic capacity is \textit{normalized} if $c(B^{2n}(1),\omega_0)=c(Z^{2n}(1),\omega_0)$.

To see how Gromov's result relates to symplectic capacities, given a symplectic manifold $(M,\omega)$, define the following quantity:
$$c_{Gr}(M,\omega):=\sup\{r\mid (B^{2n}(r),\omega_0) \textup{ can be symplectically embedded into }(M,\omega)\},$$
where $\omega_0:=\sum_{i=1}^{n}dx_i\wedge dy_i$ is the standard symplectic form in $\mathbb{R}^{2n}$. The quantity $c_{Gr}(M,\omega)$ is a positive real number, possibly infinite and it is usually referred to as the \textit{Gromov width} of the manifold $(M,\omega)$. Taking Gromov's result for granted, it is not hard to see that the Gromov width is a symplectic capacity. Furthermore, we can easily see that for any normalized symplectic capacity $c$, we have that $c_{Gr}(M,\omega)\leq c(M,\omega)$. 

Many other examples of symplectic capacities have been given in the last three decades, see \cite{Cieliebak2005QuantitativeSG} for an account on the subject. One of particular significance is the \textit{cylindrical capacity} $c_Z$, defined as follows:
$$c_Z(M,\omega):=\inf\{r\mid (M,\omega) \textup{ can be symplectically embedded into } (Z^{2n}(r),\omega_0)\}.$$
This is well-defined, at least whenever $(M,\omega)$ is a domain in $\mathbb{R}^{2n}$. It is not hard to see that for any normalized symplectic capacity $c$, we get that $c(M,\omega)\leq c_Z(M,\omega)$.

One particular problem has driven the study of symplectic capacities forward in the last three decades, the \textit{Viterbo conjecture} \cite{Viterbo2000MetricAI}. 

\begin{conj}[Viterbo, 2000]\label{conj: viterbo}
For any symplectic capacity $c$, and any convex
domain $X \subset \mathbb{R}^{2n}$
$$c^n(X) \leq n!\textup{Vol}(X).$$
\end{conj}

It was shown in \cite{ArtsteinAvidan2013FromSM}, that the famous Mahler’s conjecture on the volume
product of centrally symmetric convex bodies, which had remained open for the
past 80 years, and has been proven for dimension less than 4, is equivalent to
the restriction of Conjecture \ref{conj: viterbo} to Lagrangian products $K\times_L K^{\circ}$ of centrally symmetric
convex bodies $K \subset\mathbb{R}^n$ and their dual bodies $K^{\circ} \subset \mathbb{R}^n$, where $K^{\circ}$ is the \textit{polar dual body} of $K\subset \mathbb{R}^n$, defined as
$$K^{\circ}:=\{y\in \mathbb{R}^n\mid\sup_{x\in K} |\langle x,y\rangle|\leq 1\},$$
and given $K\subset{R}^n_{(x_1,\ldots,x_n)}$ and $T\subset{R}^n_{(y_1,\ldots,y_n)}$, we define the \textit{Lagrangian product} $K\times_L T$ to be the usual Cartesian product, with symplectic form $\sum_{i=1}^ndx_i\wedge dy_i$.

Recently, Haim-Kislev and Ostrover \cite{HaimKislev2024ACT} provided a counterexample to the Viterbo conjecture. They computed the Hofer-Zehnder capacity of the Lagrangian product of a regular pentagon and its $90^{\circ}$ rotation, and showed that for this capacity, Conjecture \ref{conj: viterbo} does not hold.

Although Conjecture \ref{conj: viterbo} is not valid for all convex domains in $\mathbb{R}^{2n}$, in view of its connection to the Mahler conjecture, it is still an interesting open question whether it holds for the class of Lagrangian products $K\times_L K^{\circ}$, where $K\subset\mathbb{R}^n$ is a centrally symmetric convex body and $K^{\circ}$ is the polar dual body of $K\subset \mathbb{R}^n$.

Very little progress has been done in this restricted version of the Viterbo conjecture. The conjecture has been shown to be true when taking $K$ to be the unit ball in the $p$-norm, see \cite{Karasev2019MahlersCF}. 

In this paper, we analyze for certain class of symmetrically convex bodies in $\mathbb{R}^{2n}$, two different notions of dual body and study the symplectic capacities of the two different Lagrangian products induced by these two notions. Our goal is to discuss some classes of domains in $\mathbb{R}^{2n}$ where the following conjecture holds. 

\begin{conj}[strong Viterbo conjecture]\label{conj: viterbo2}
If $X$ is a convex domain in $\mathbb{R}^{2n}$, then all
normalized symplectic capacities of $X$ agree.
\end{conj}

Notice that a counterexample to Conjecture \ref{conj: viterbo} proves that Conjecture \ref{conj: viterbo2} is also not valid for every $X$ convex. However, as pointed out before, it is interesting to understand for what class of domains it does hold. It was shown in \cite{Gutt2022ExamplesAT} that Conjecture \ref{conj: viterbo2} holds for monotone toric domains in $\mathbb{R}^4$ and compact convex domains in $\mathbb{C}^n$ invariant by the diagonal $S^1$-action. The first of these results was subsequently generalized to higher dimensions in \cite{CristofaroGardiner2023OnTA}. It is then a consequence of the results in \cite{Ostrover2019SymplecticEO} and \cite{Gutt2022ExamplesAT} that Conjecture \ref{conj: viterbo2} also holds for the Lagrangian $l^p$-sum of two disks. Our two main results give two different classes of Lagrangian products where Conjecture \ref{conj: viterbo2} holds. Before stating our results, we introduce some necessary setup.


Let $\Phi:=(\Phi_1,\ldots,\Phi_n)$ be a $n$-tuple of \textit{Young functions}, i.e. $\Phi_i:\mathbb{R}_{\geq 0} \to \mathbb{R}_{\geq 0}$ is smooth, convex, increasing, such that $\Phi_i(0)=0$ and $\lim_{x\to \infty}\Phi_i(x)=\infty$. Notice that this implies, in particular, that the inverse functions $\Phi_i^{-1}$ exist. The \textit{Luxemburg norm} induced by $\Phi$ in $\mathbb{R}^n$ is defined by:
$$||(x_1,\ldots,x_n)||_{\Phi}:=\inf\left\{\lambda>0\mid \sum_{i=1}^{n}\Phi_i\left(\frac{|x_i|}{\lambda}\right)\leq 1\right\}.$$
We define the unit ball $K_{\Phi}\subset\mathbb{R}^n$ for this norm, by the set of points $(x_1,\ldots,x_n)\in \mathbb{R}^n$ such that 
$$\Phi_1(|x_1|)+\ldots+\Phi_n(|x_n|)\leq 1.$$
Let $\Phi^*:=(\Phi^*_1,\ldots,\Phi^*_n)$ with $\Phi^*_i:\mathbb{R}_{\geq 0} \to \mathbb{R}_{\geq 0}$, be the \textit{Legendre transform} of $\Phi$, i.e. 
\begin{equation}\label{eq: legendre_transform}
\Phi^*_i(y)=\sup_{x\in \mathbb{R}_{\geq 0}}(xy-\Phi_i(x)), \hspace{5mm} \textup{for every }i=1,\ldots,n.
\end{equation}
Then, $\Phi_i^*$ is also a Young function, and so, we can define also a unit ball associated to $\Phi^*$. From equation \eqref{eq: legendre_transform} is evident that the pair $(\Phi_i,\Phi_i^*)$ satisfies \textit{Young's inequality}:
\begin{equation}\label{eq: young}
xy\leq \Phi_i(x)+\Phi_i^*(y),
\end{equation}
with equality if, and only if, $y=\Phi_i'(x)$. In particular, it follows from Young's inequality, that 
\begin{equation}\label{eq: young_2}
\Phi_i^{-1}(x)(\Phi_i^*)^{-1}(y)\leq x+y,
\end{equation}
for every $x,y\in \mathbb{R}_{> 0}$. Furthermore, we also have the bound
\begin{equation}\label{eq: lower_ineq_Young}
x<\Phi_i^{-1}(x)(\Phi_i^*)^{-1}(x),
\end{equation}
see \cite[Proposition 2.1]{Rao1991TheoryOO}. We include in this note the proof of \eqref{eq: lower_ineq_Young} for the sake of completeness, although we will defer it to Section \ref{sec: inequality}.

We are now ready to state the first of our two main results. 

\begin{theorem}\label{thm: capacity}
Let $c$ be any normalized symplectic capacity. Let $K_{\Phi}\subset \mathbb{R}^n$ be the unit ball of a $\Phi$-norm and $K_{\Phi^*}\subset \mathbb{R}^n$ be the unit ball of the $\Phi^*$-norm. Then
$$c(K_{\Phi}\times_L K_{\Phi^*})=4\min_i(\Phi_i^{-1}(1){(\Phi^*_i)}^{-1}(1)).$$
In particular, for any Young function $\Phi$, we have that $c(K_{\Phi}\times_L K_{\Phi^*})\leq 8$.
\end{theorem}

Notice that we can then obtain a Mahler-like statement from Theorem \ref{thm: capacity}. Indeed, since Conjecture \ref{conj: viterbo} can be easily seen to hold for the Gromov width, which by Theorem \ref{thm: capacity} is equal to $4\min_i(\Phi_i^{-1}(1){(\Phi^*_i)}^{-1}(1))$, we then have the relation
$$\frac{4^n}{n!}\leq \frac{\textup{Vol}(K_{\Phi})\textup{Vol}(K_{\Phi^*})}{\left(\min_i(\Phi_i^{-1}(1){(\Phi^*_i)}^{-1}(1))\right)^n}<\textup{Vol}(K_{\Phi})\textup{Vol}(K_{\Phi^*}),$$
by simple use of the property \eqref{eq: lower_ineq_Young}.

As an application of this theorem, we recover a result by Karasev in \cite{Karasev2019MahlersCF}.

\begin{cor}
Let $D_p(1)$ be the unit ball in the $p$-norm, for $1<p<\infty$. Then, for any normalized symplectic capacity $c$, we get that $c(D_p(1)\times_L D_q(1))=4$, where $q$ is such that $1/p+1/q=1$.
\end{cor}
\begin{proof}
Let $p>1$ and for every $i=1,\ldots,n$, take $\Phi_i(t)=\frac{t^p}{p}$. Then, it can be shown that the Legendre transform of $\Phi_i$ is given by $\Phi^*_i(t)=\frac{t^q}{q}$, for $q$ such that $1/p+1/q=1$. By Theorem \ref{thm: capacity}, we have that $c(K_{\Phi}\times_L K_{\Phi^*})=4p^{\frac{1}{p}}q^{\frac{1}{q}}$. It is not hard to see that 
$$K_{\Phi}\times_L K_{\Phi^*}\overset{s}{\cong}p^{\frac{1}{2p}}q^{\frac{1}{2q}}(D_p(1)\times_L D_q(1)).$$
This comes from the observation that $K_{\Phi}\times_L K_{\Phi^*}= \left(p^{1/p}D_p(1)\right)\times_L \left(q^{1/q}D_q(1)\right)$, and the symplectomorphism is given by the rescaling 
$$(x_1,\ldots,x_n,y_1,\ldots,y_n) \to \left(\frac{q^{1/2q}}{p^{1/2p}}x_1,\ldots,\frac{q^{1/2q}}{p^{1/2p}}x_n,\frac{p^{1/2p}}{q^{1/2q}}y_1,\ldots,\frac{p^{1/2p}}{q^{1/2q}}y_n\right).$$
This implies, by the conformality property of symplectic capacities, that 
\[
\scalebox{0.9}{$
4p^{\frac{1}{p}}q^{\frac{1}{q}}=c\left(K_{\Phi}\times_L K_{\Phi^*}\right)=\left(p^{\frac{1}{2p}}q^{\frac{1}{2q}}\right)^2c\left(D_p(1)\times_L D_q(1)\right).
$}
\]
Therefore, $c\left(D_p(1)\times_L D_q(1)\right)= 4$.
\end{proof}

\begin{example}\label{ex: Gaussian}
For every $i=1,\ldots,n$, take $\Phi_i(t)=\frac{e^t-1}{e}$. The function $\Phi_i$ is of particular interest in the theory of Luxemburg and Orlicz normed spaces. It can be shown that the Legendre transform of $\Phi_i$ is given by $\Phi^*_i(t)=t\log t+\frac{1}{e}$. Notice that $(\Phi_i^*)^{-1}(t)=e^{W(t-\frac{1}{e})}$, where $W(t)$ is the Lambert $W$ function. By Theorem \ref{thm: capacity}, we have that  $c(K_{\Phi}\times_L K_{\Phi^*})=4e^{W(1-\frac{1}{e})}\log (e+1)$. See Figure \ref{fig:fig1} and \ref{fig:fig2}. 
\end{example}

\begin{figure}[h]
    \centering
    \begin{subfigure}{0.3\textwidth}
        \centering
        \includegraphics[width=\textwidth]{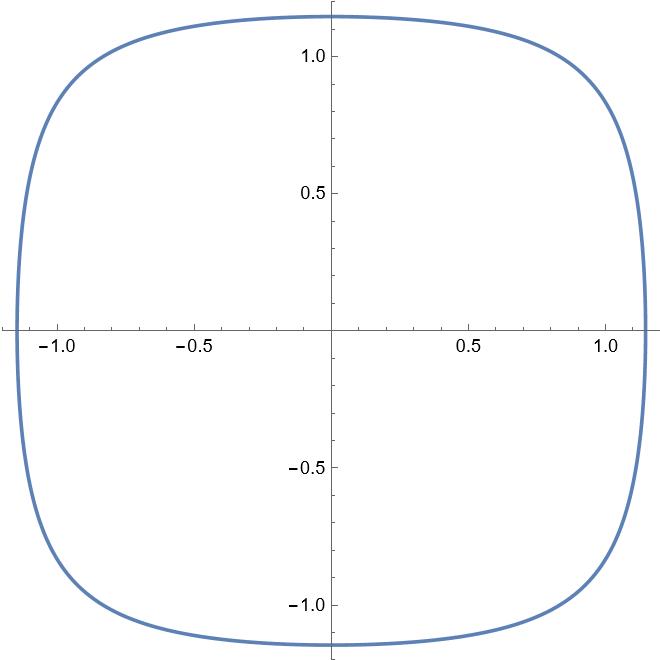}
        \caption{$K_{\Phi}$}
        \label{fig:fig1}
    \end{subfigure}
    \hfill
    \begin{subfigure}{0.3\textwidth}
        \centering
        \includegraphics[width=\textwidth]{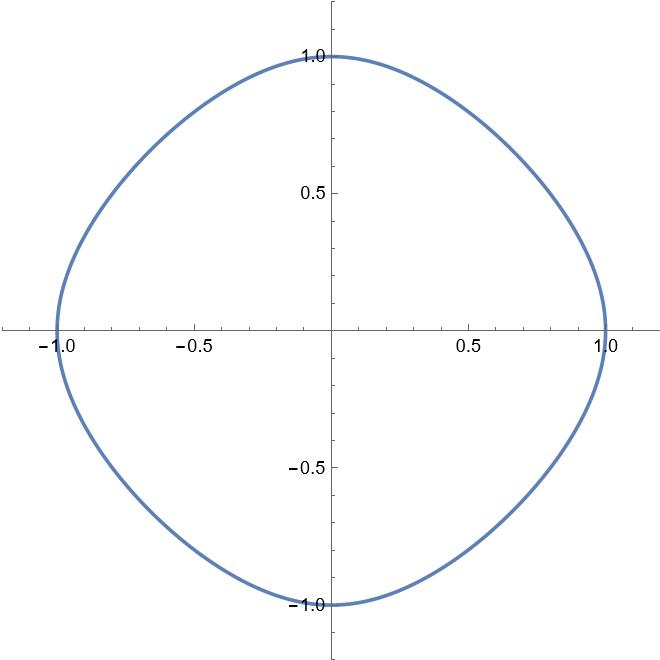}
        \caption{$K_{\Phi^*}$}
        \label{fig:fig2}
    \end{subfigure}
    \hfill
    \begin{subfigure}{0.3\textwidth}
        \centering
        \includegraphics[width=\textwidth]{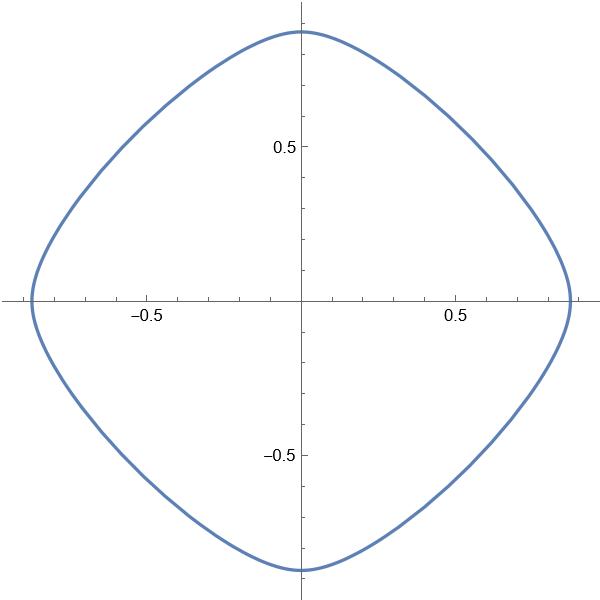}
        \caption{$K_{\Phi}^{\circ}$}
        \label{fig:fig3}
    \end{subfigure}
    \caption{The balls $K_{\Phi},K_{\Phi^*}$ and $K_{\Phi}^{\circ}$, for $\Phi_i(t)=\frac{1}{e}(e^t-1)$, for every $i=1,2$.}\label{fig: Gaussian_Product}
    
\end{figure}

We now state our second main result.
\begin{theorem}\label{thm: capacity2}
Let $c$ be any normalized symplectic capacity. Let $K_{\Phi}\subset \mathbb{R}^n$ be the unit ball of a $\Phi$-norm and $K_{\Phi}^{\circ}\subset \mathbb{R}^n$ its polar dual body. Then we have that
$$2\min_i(\Phi_i^{-1}(1){(\Phi^*_i)}^{-1}(1))\leq c(K_{\Phi}\times_L K_{\Phi}^{\circ})\leq 4.$$
In particular, we have that $c(K_{\Phi}\times_L K_{\Phi}^{\circ})> 2$.
\end{theorem}
The most important application of this result is the proof of the strong Viterbo conjecture for a particular class of Lagrangian products.

\begin{cor}
Let $c$ be any normalized symplectic capacity. Let $K_{\Phi}\subset \mathbb{R}^n$ be the unit ball of a $\Phi$-norm and $K_{\Phi}^{\circ}\subset \mathbb{R}^n$ its polar dual body. Assume $\min_i(\Phi_i^{-1}(1){(\Phi^*_i)}^{-1}(1))=\Phi_I^{-1}(1){(\Phi^*_I)}^{-1}(1)$ for some index $I\in\{1,\ldots,n\}$ and $\Phi_I'(1)=1$, then $c(K_{\Phi}\times_L K_{\Phi}^{\circ})=4$.
\end{cor}
\begin{proof}
If $\Phi_I'(1)=1$, it follows from equality in \eqref{eq: young_2}, that $\Phi_I^{-1}(1){(\Phi^*_I)}^{-1}(1)=2$. Therefore, by Theorem \ref{thm: capacity2} we have that 
$$4=2\Phi_I^{-1}(1){(\Phi^*_I)}^{-1}(1)=2\min_i(\Phi_i^{-1}(1){(\Phi^*_i)}^{-1}(1))\leq c(K_{\Phi}\times_L K_{\Phi}^{\circ})\leq 4.$$
The claim follows.
\end{proof}

\begin{remark}
Notice that rescaling the $n$-tuple $\Phi$ by a constant $c>0$ will change the domain $K_{\Phi}$ into another domain $K_{c\Phi}$, which is not necessarily a rescaling of the $K_{\Phi}$. This is easily seen for example, for the 2-tuple $\Phi(t)=(\Phi_1(t),\Phi_2(t)):=(t^2,e^t-1)$. See Figure \ref{fig: remark_before} for the Lagrangian product $K_{\Phi}\times_L K_{\Phi}^{\circ}$. Notice in this case, $\min_i(\Phi_i^{-1}(1){(\Phi^*_i)}^{-1}(1))=\Phi_2^{-1}(1){(\Phi^*_2)}^{-1}(1)$ and $\Phi_2'(1)=e$. So, we can consider the rescaling of $\Phi$ given by $\overline{\Phi}(t)=(\overline{\Phi}_1(t),\overline{\Phi}_2(t)):=\left(\frac{1}{e}t^2,\frac{1}{e}(e^t-1)\right)$. See Figure \ref{fig: remark_after} for the Lagrangian product $K_{\overline{\Phi}}\times_L K_{\overline{\Phi}}^{\circ}$.
\end{remark}

\begin{figure}[h]
    \centering
    \begin{subfigure}{0.4\textwidth}
        \centering
        \includegraphics[width=\textwidth]{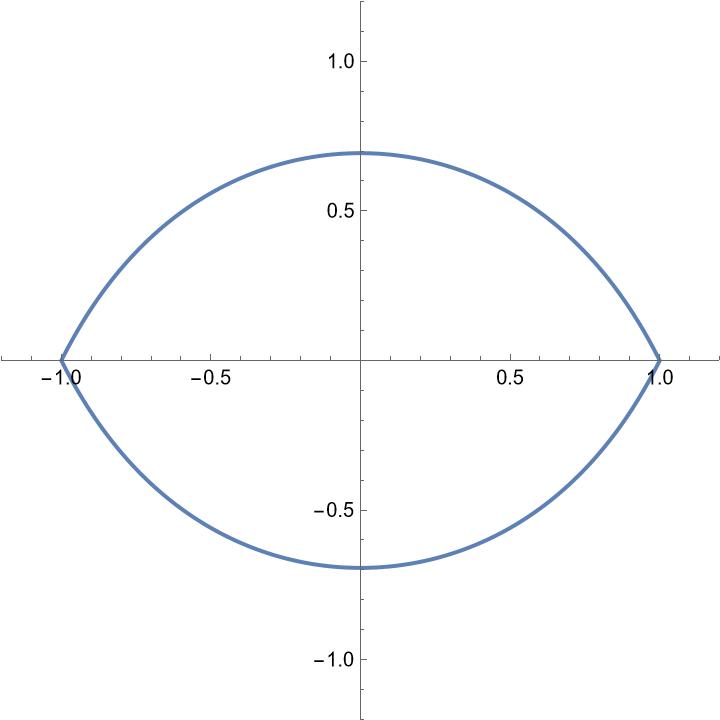}
        \caption{$K_{\Phi}$}
    \end{subfigure}
    \hfill
    \begin{subfigure}{0.4\textwidth}
        \centering
        \includegraphics[width=\textwidth]{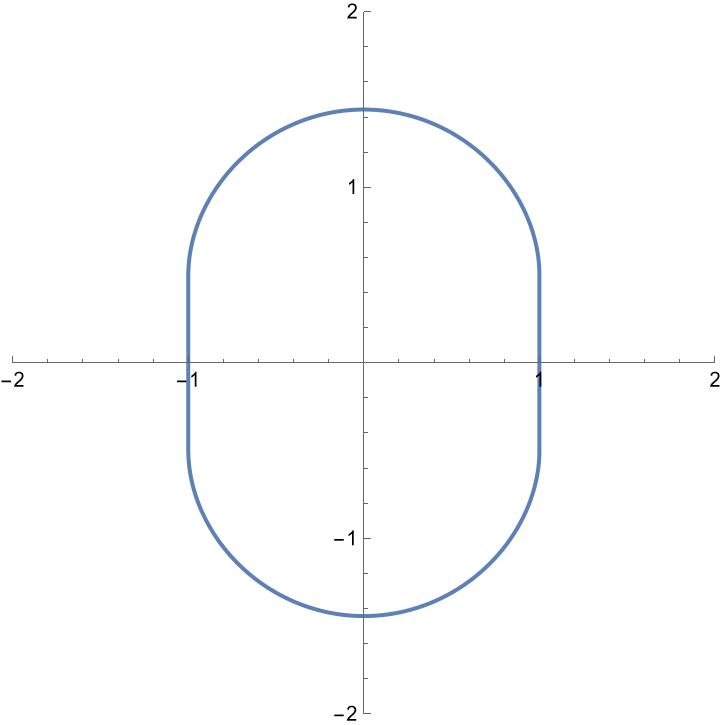}
        \caption{$K_{\Phi}^{\circ}$}
    \end{subfigure}
    \caption{The Lagrangian product $K_{\Phi}\times K_{\Phi}^{\circ}$ for $\Phi(t)=(\Phi_1(t),\Phi_2(t)):=(t^2,e^t-1)$.}\label{fig: remark_before}
    
\end{figure}
\begin{figure}[h]
    \centering
    \begin{subfigure}{0.4\textwidth}
        \centering
        \includegraphics[width=\textwidth]{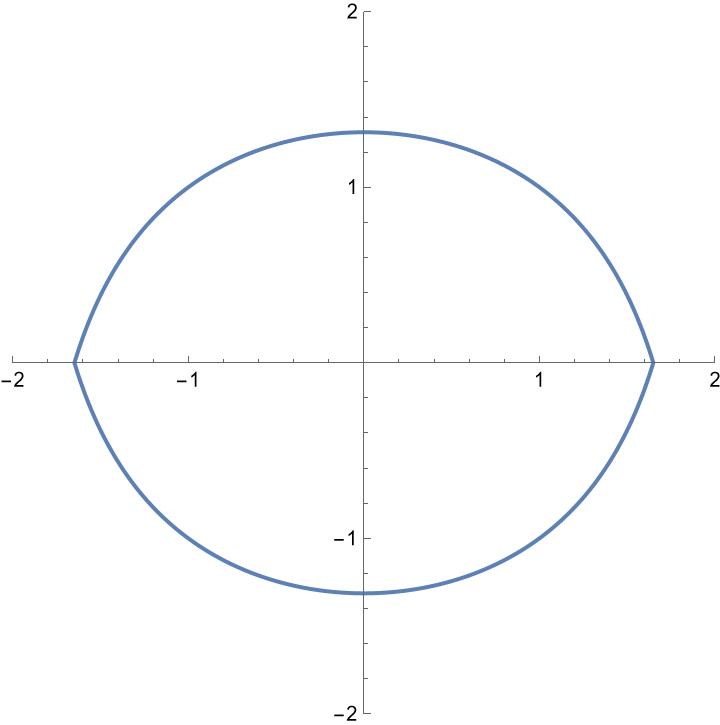}
        \caption{$K_{\overline{\Phi}}$}
    \end{subfigure}
    \hfill
    \begin{subfigure}{0.4\textwidth}
        \centering
        \includegraphics[width=\textwidth]{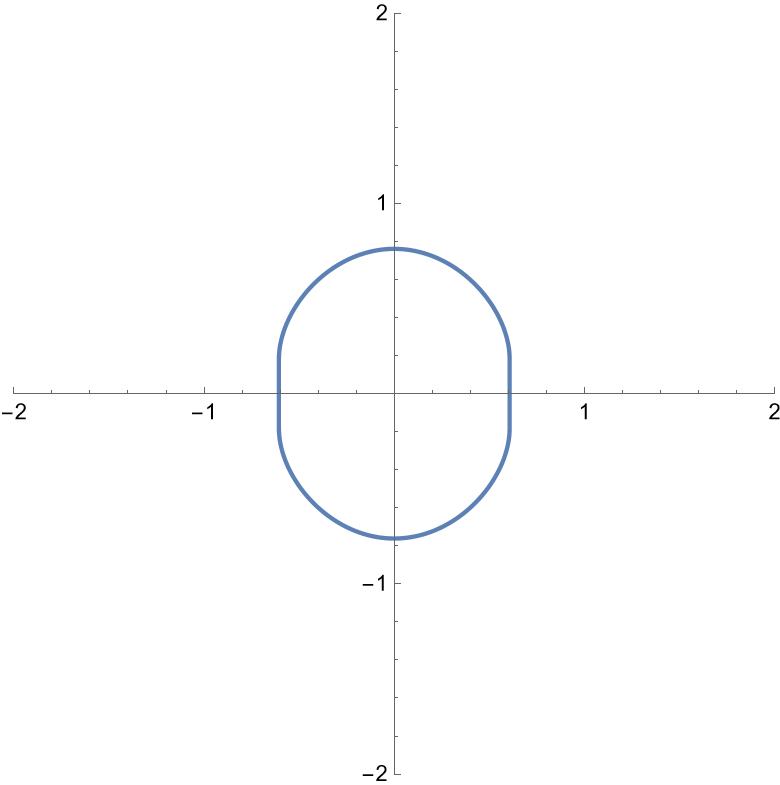}
        \caption{$K_{\overline{\Phi}^*}$}
    \end{subfigure}
    \caption{The Lagrangian product $K_{\overline{\Phi}}\times K_{\overline{\Phi}}^{\circ}$ for $\overline{\Phi}(t)=(\overline{\Phi}_1(t),\overline{\Phi}_2(t)):=\left(\frac{1}{e}t^2,\frac{1}{e}(e^t-1)\right)$.}\label{fig: remark_after}
    
\end{figure}

We see in the next result, sufficient conditions that $\Phi$ should satisfy, when $\Phi$ is of polynomial type, in order for the strong Viterbo conjecture to hold. 
\begin{cor}\label{cor: polynomial}
Let $c$ be any normalized symplectic capacity. Let $K_{\Phi}\subset \mathbb{R}^n$ be the unit ball of a $\Phi$-norm, where for every $i=1\ldots,n$, we have that $\Phi_i(t)=a_kt^k+\ldots+a_1t$ is a polynomial, with $a_1,\ldots, a_k\in \mathbb{R}$, such that $a_1\geq 0, \sum_{i=1}^kia_i=1$ and $\Phi_i''(t)\geq 0$, for all $t\geq 0$. Then $c(K_{\Phi}\times_L K_{\Phi}^{\circ})=4$.
\end{cor}
\begin{proof}
Notice that $\Phi_i'(0)=a_1\geq 0$ and $\Phi_i''(t)\geq 0$, imply that $\Phi_i(t)$ is non-decreasing and convex, for all $t\geq 0$. Furthermore, $\Phi_i'(1)=\sum_{i=1}^kia_i=1$. Hence, by Corollary \ref{cor: polynomial}, the claim follows.
\end{proof}

\begin{example}
Note that $\Phi_i$ from Example \ref{ex: Gaussian} satisfy $\Phi_i'(1)=1$. By Theorem \ref{thm: capacity2}, we have that  $c(K_{\Phi}\times K_{\Phi}^{\circ})=4$. 
\end{example}

\begin{remark}
In the spirit of Theorems \ref{thm: capacity} and \ref{thm: capacity2}, there is one more notion of duality in Lagrangian products, induced by Young functions. That is, we can consider the Lagrangian product $K_{\Phi}\times_LK_{\Phi^*}^{\circ}$, given by taking the second factor to be the polar dual body of the unit ball for the Legendre transform $\Phi^*$ of $\Phi$. We think it would also be interesting to understand the behaviour of normalized symplectic capacities in this class of domains. However, we do not know if the techniques used in this paper apply directly to this case.
\end{remark}


\subsection{Idea of the proof}

The idea of the proofs of Theorems \ref{thm: capacity} and \ref{thm: capacity2} is as follows, for the upper bound we use a standard argument found in \citep{ArtsteinAvidan2013FromSM}, where we construct a symplectic projection map and then use Gromov's Non-Squeezing Theorem. For the lower bound, we construct a symplectic embedding of balls of capacities $4\min_i(\Phi_i^{-1}(1){(\Phi^*_i)}^{-1}(1))$ and $2\min_i(\Phi_i^{-1}(1){(\Phi^*_i)}^{-1}(1))$, into $K_{\Phi}\times K_{\Phi^*}$ and $K_{\Phi}\times K_{\Phi}^{\circ}$, respectively. This will give us the desired lower bound. This is the content of the following two results.

\begin{theorem}\label{thm: embedding}
Let $K_{\Phi}\subset \mathbb{R}^n$ be the unit ball of a $\Phi$-norm and $K_{\Phi^*}\subset \mathbb{R}^n$ be the unit ball of the $\Phi^*$-norm. Let $c:=4\min_i(\Phi_i^{-1}(1){(\Phi^*_i)}^{-1}(1))$. Then, for any $\varepsilon>0$ there exists a symplectic embedding of  $B^{4}(c)$ into $(1+O(\varepsilon))(K_{\Phi}\times_L K_{\Phi^*})$. 
\end{theorem}

\begin{theorem}\label{thm: embedding2}
Let $K_{\Phi}\subset \mathbb{R}^n$ be the unit ball of a $\Phi$-norm. Let $c:=4\min_i(\Phi_i^{-1}(1){(\Phi^*_i)}^{-1}(1))$. Then, for any $\varepsilon>0$ there exists a symplectic embedding of  $B^{4}(c)$ into $\sqrt{2}(1+O(\varepsilon))(K_{\Phi}\times_L K_{\Phi}^{\circ})$. 
\end{theorem}


The construction of these embeddings follows an idea first exposed in \cite{Latschev2011TheGW} and also used in \cite{Karasev2019MahlersCF}. We construct the symplectic embedding by pairing together certain area preserving embeddings.

\noindent{\bf Structure of the paper:} In Section \ref{sec: construction}, we prove Theorems \ref{thm: embedding} and \ref{thm: embedding2}. In Section \ref{sec: capacity}, we give the proof of Theorems \ref{thm: capacity} and \ref{thm: capacity2}. In Section \ref{sec: inequality} we give a self-contained proof of the inequality in \eqref{eq: lower_ineq_Young}.

\noindent{\bf Acknowledgements:} The author was supported by the ISF Grant No. 2445/20.

\section{Construction of the embedding}\label{sec: construction}
In this section we show the proofs of Theorems \ref{thm: embedding} and \ref{thm: embedding2}, that is, we construct a symplectic embedding of a ball into $K_{\Phi}\times K_{\Phi^*}$.  
  
\begin{proof}[Proof of Theorem \ref{thm: embedding}]
Let $\varepsilon>0$ and $c_i=4\Phi_i^{-1}(1){(\Phi^*_i)}^{-1}(1)$, for every $i=1,\ldots,n$.  Notice that $c_i\leq 8$ for every $i$, by \eqref{eq: young_2}. Choose a family of area preserving embeddings 
$$\sigma_i:B^2(c_i)\to (-\Phi_i^{-1}(1)-\varepsilon,\Phi_i^{-1}(1)+\varepsilon) \times\left(-{(\Phi^*_i)}^{-1}(1)-\varepsilon,{(\Phi^*_i)}^{-1}(1)+\varepsilon\right),$$
such that:
\begin{enumerate}[label=\roman*.]
\item \label{eq: ineq1} $\Phi_i(|x(\sigma_i(z))|)<\frac{1}{c_i}\pi|z|^2+\frac{\varepsilon}{n}, \hspace{5mm} \forall z\in B^2(4),$
\item \label{eq: ineq2} $\Phi^*_i(|y(\sigma_i(z))|)<\frac{1}{c_i}\pi|z|^2+\frac{\varepsilon}{n}, \hspace{5mm} \forall z\in B^2(4),$
\end{enumerate}
for every $i=1,\ldots,n$.

These inequalities can be achieved by area preserving maps because they mean that the disc of radius $r$, centered at the
origin, must be mapped into a rectangle of area slightly larger than $\pi r^2$. Indeed, using the assumptions above and the fact that both $\Phi_i$ and $\Phi^*_i$ are increasing, we have the inequality
$$|x(\sigma_i(z))y(\sigma_i(z))|\leq \Phi_i^{-1}\left(\frac{1}{c_i}\pi|z|^2+\frac{\varepsilon}{n}\right){(\Phi^*_i)}^{-1}\left(\frac{1}{c_i}\pi|z|^2+\frac{\varepsilon}{n}\right).$$
By the bounds in \eqref{eq: young_2} and \eqref{eq: lower_ineq_Young}, we have that 
$$\frac{1}{c_i}\pi|z|^2+\frac{\varepsilon}{n}<\Phi_i^{-1}\left(\frac{1}{c_i}\pi|z|^2+\frac{\varepsilon}{n}\right){(\Phi^*_i)}^{-1}\left(\frac{1}{c_i}\pi|z|^2+\frac{\varepsilon}{n}\right)\leq 2\left(\frac{1}{c_i}\pi|z|^2+\frac{\varepsilon}{n}\right),$$
which implies that 
$$4|x(\sigma_i(z))y(\sigma_i(z))|\leq 2\left(\frac{4}{c_i}\pi|z|^2+4\frac{\varepsilon}{n}\right),$$
which proves the consistency of the areas. The precise construction of such embeddings can be found in \cite{Schlenk2001ESchlembeddingPI} and \cite{Schlenk2005ESchlembeddingPI}.

\begin{figure}[htbp] 
    \centering 
\begin{tikzpicture}

\begin{scope}
\foreach \r/\c [count=\i] in {2.0/none, 1.4/gray!20, 0.8/gray!60, 0.4/gray!100} {
  \ifnum\i=1
    \draw (0,0) circle (\r); 
  \else
    \draw[fill=\c, draw=none] (0,0) circle (\r);
    \draw[dashed] (0,0) circle (\r);
  \fi
}
\draw (0,0) node[circle,inner sep=1.5pt,fill=black] {};
\draw[->] (-2.5,0) -- (2.5,0) node at (2.6,0.2) {$x$} node at (2.15,0.3) {$1$};
\draw[->] (0,-2.5) -- (0,2.5) node at (0.2,2.6) {$y$};

\end{scope}

\draw[->,thick] (3,0) -- (4.5,0) node at (3.8,0.2) {$\sigma_i$};

\begin{scope}[shift={(7.5,0)}]
\foreach \r/\c [count=\i] in {2.0/none, 1.4/gray!20, 0.8/gray!60, 0.4/gray!100} {
  \pgfmathsetmacro{\yfactor}{1.2}
  \pgfmathsetmacro{\ry}{\r*\yfactor}
  \ifnum\i=1
    \draw (-\r,-\ry) rectangle (\r,\ry); 
  \else
    \draw[fill=\c, draw=none, rounded corners=0.3cm] (-\r,-\ry) rectangle (\r,\ry);
    \draw[dashed, rounded corners=0.3cm] (-\r,-\ry) rectangle (\r,\ry);
  \fi
}
\draw (0,0) node[circle,inner sep=1.5pt,fill=black] {};
\draw[->] (-2.5,0) -- (2.5,0) node at (2.6,0.2) {$x$} node at (3.2,-0.35) {$\Phi_i^{-1}(1+\varepsilon/n)$};
\draw[->] (0,-2.7) -- (0,2.7) node at (0.2,2.8) {$y$} node at (-1.35,2.7) {${(\Phi^*_i)}^{-1}(1+\varepsilon/n)$}; 
\end{scope}

\end{tikzpicture}
\caption{The map $\sigma_i$.} 
    
\end{figure}
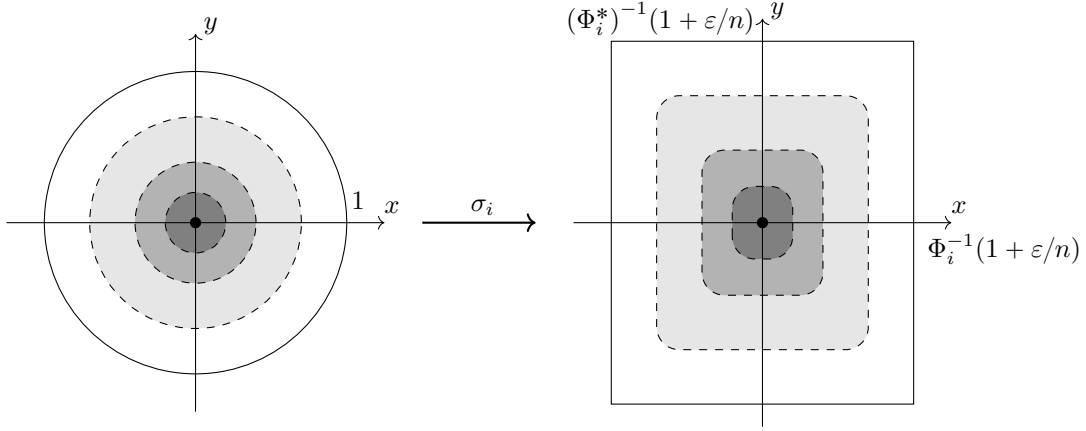
We notice now that the symplectic embedding $\sigma_1\times\ldots\times\sigma_n: B^2(c_1) \times\ldots\times B^2(c_n) \to \mathbb{R}^{2n}$ maps $B^{2n}(c)$ into $(1+O(\varepsilon))(K_{\Phi}\times_L K_{\Phi^*})$. Indeed, for $(z_1,\ldots,z_n)\in B^{2n}(c)$ we have $\pi(|z_1|^2+\ldots+|z_n|^2) \leq c$. Together with item \eqref{eq: ineq1} we can estimate
\begin{equation*}
\begin{split}
 \sum_{i=1}^n\Phi_i(|x_i\left((\sigma_1\times\ldots\times \sigma_n)(z_1,\ldots,z_n)\right)|)&=\sum_{i=1}^n\Phi_i(|x_i(\sigma(z_i))|),\\
 &\leq \frac{\pi}{c_i}|z_1|^2+\ldots+\frac{\pi}{c_n}|z_n|^2+\varepsilon,\\
&\leq \frac{\pi}{c}(|z_1|^2+\ldots+|z_n|^2)+\varepsilon,\\
&\leq 1+\varepsilon.
\end{split}
\end{equation*}
Similarly, using item \eqref{eq: ineq2} we can estimate
\begin{equation*}
\begin{split}
 \sum_{i=1}^n\Phi^*_i(|y_i\left((\sigma_1\times\ldots\times \sigma_n)(z_1,\ldots,z_n)\right)|)&=\sum_{i=1}^n\Phi^*_i(|y_i(\sigma(z_i))|),\\
  &\leq \frac{\pi}{c_i}|z_1|^2+\ldots+\frac{\pi}{c_n}|z_n|^2+\varepsilon,\\
&\leq\frac{\pi}{c}(|z_1|^2+\ldots+|z_n|^2)+\varepsilon,\\
&\leq 1+\varepsilon.
\end{split}
\end{equation*}
This concludes the proof.
\end{proof}

\begin{proof}[Proof of Theorem \ref{thm: embedding2}]
The proof carries out almost exactly as the proof of Theorem \ref{thm: embedding}. We choose, the symplectic embedding $\sigma_1\times\ldots \times \sigma_n$ in the same manner.

We notice now that the symplectic embedding $\sigma_1\times\ldots\times\sigma_n: B^2(c_1) \times\ldots\times B^2(c_n) \to \mathbb{R}^{2n}$ maps $B^{2n}(c)$ into $\sqrt{2}(1+O(\varepsilon))(K_{\Phi}\times_L K_{\Phi}^{\circ})$. Indeed, for $(z_1,\ldots,z_n)\in B^{2n}(c)$ we have $\pi(|z_1|^2+\ldots+|z_n|^2) \leq c$. We can estimate, again, that
\begin{equation*}
\sum_{i=1}^n\Phi_i(|x_i\left((\sigma_1\times\ldots\times \sigma_n)(z_1,\ldots,z_n)\right)|)\leq 1+\varepsilon
\end{equation*}
Let $\mathbf{x}=(x_1,\ldots,x_n)\in K_{\Phi}$, using item \eqref{eq: ineq2} and Young's inequality in \eqref{eq: young}, we can estimate
\begin{equation*}
\begin{split}
 |\langle\mathbf{y}(\left(\sigma_1\times\ldots\times \sigma_n)(z_1,\ldots,z_n)\right),\mathbf{x}\rangle|&=\bigg|\sum_{i=1}^n y_i(\sigma_i(z_i))x_i\bigg|,\\
&\leq\sum_{i=1}^n |y_i(\sigma_i(z_i))x_i|,\\
&\leq \sum_{i=1}^n \left(\Phi_i^*(|y_i(\sigma_i(z_i))|)+\Phi_i(|x_i|)\right),\\
&\leq \sum_{i=1}^n\frac{\pi}{c_i}|z_i|^2+\varepsilon+\sum_{i=1}^n \Phi_i(|x_i|),\\
& \leq\frac{\pi}{c}(|z_1|^2+\ldots+|z_n|^2)+\varepsilon+\sum_{i=1}^n \Phi_i(|x_i|),\\
&\leq 2+\varepsilon.
\end{split}
\end{equation*}
The proof can be concluded after applying a (symplectic) rescaling.
\end{proof}

\section{Proof of the main theorems}\label{sec: capacity}

We are finally ready to give the proof of our main results, Theorems \ref{thm: capacity} and \ref{thm: capacity2}.

\begin{proof}[Proof of Theorem \ref{thm: capacity}]
By Theorem \ref{thm: embedding}, monotonicity and conformality of symplectic capacities, we have that 
$$4\min_i(\Phi_i^{-1}(1){(\Phi^*_i)}^{-1}(1))=c(B^{2n}(4\min_i(\Phi_i^{-1}(1){(\Phi^*_i)}^{-1}(1))))\leq (1+O(\varepsilon))^2c(K_{\Phi}\times_L K_{\Phi^*}),$$
for any $\varepsilon>0$ small enough. Therefore, we immediately have that 
$$4\min_i(\Phi_i^{-1}(1){(\Phi^*_i)}^{-1}(1))\leq c(K_{\Phi}\times_L K_{\Phi^*}).$$
Let $I\in\{1,\ldots,n\}$ be the index where $\min_i(\Phi_i^{-1}(1){(\Phi^*_i)}^{-1}(1))$ is attained, consider the symplectic projection map $f_I:K_{\Phi}\times_L K_{\Phi^*}\to \mathbb{R}^2$, given by
$$f(x_1,\ldots,x_n,y_1,\ldots,y_n)=(x_I,y_I).$$
Notice that the image of $f$ is contained in 
\[
 [-\Phi_I^{-1}(1),\Phi_I^{-1}(1)]\times[-{(\Phi^*_I)}^{-1}(1),{(\Phi^*_I)}^{-1}(1)]\times \mathbb{R}^{2n-2},
\]
 By arguments involving Gromov's Non-Squeezing Theorem, we can see that 
$$c(K_{\Phi}\times_L K_{\Phi^*})\leq 4\min_i(\Phi_i^{-1}(1){(\Phi^*_i)}^{-1}(1))),$$
and this concludes the proof.
\end{proof}

\begin{proof}[Proof of Theorem \ref{thm: capacity2}]
The proof is similar in nature to that of Theorem \ref{thm: capacity}. By projecting to a smart choice of a symplectic plane, it follows from Gromov's Non-Squeezing Theorem, that for any normalized symplectic capacity, $c(K_{\Phi}\times_L K_{\Phi}^{\circ}) \leq 4$. On the other hand, the lower bound comes as a direct consequence of Theorem \ref{thm: embedding2} and the conformality of symplectic capacities. In particular, by \eqref{eq: lower_ineq_Young} we have that $ c(K_{\Phi}\times_L K_{\Phi}^{\circ})> 2$, as claimed.
\end{proof}

\section{A technical inequality}\label{sec: inequality}

In this section we present the proof of the inequality in \eqref{eq: lower_ineq_Young}. This is extracted from \cite{Rao1991TheoryOO}.

\begin{theorem}
Let $\Phi:(a,b)\to \mathbb{R}$ be a function. Then $\Phi$ is convex, if and only if, for each closed subinterval $[c,d]\subset (a,b)$ we have that 
\begin{equation}\label{eq: integral_theorem}
\Phi(x)=\Phi(c)+\int_c^x\varphi(t)dt,\hspace{8mm} c\leq x\leq d,
\end{equation}
where $\varphi:\mathbb{R}\to\mathbb{R} $ is a monotone, nondecreasing and left continuous function. Also, $\Phi$ has a left and a right derivative at each point of $(a, b)$ and they are equal except perhaps for at most a countable number of points.
\end{theorem}
\begin{proof}
Let $c \leq x_i \leq d, i = 1,2$, and, in order to use the definition of convexity, let $c\leq x < y < z \leq d$. Making $\alpha = \frac{y-x}{z-x}$ and $\beta =\frac{z- y}{z-x}$, we have that $0 <\alpha,\beta < 1$ and $\alpha + \beta = 1$. By taking $x_1 = z$ and $x_2 = x$, we can see that 
$$\Phi(y) = \Phi(\alpha x_1 + \beta x_2)\leq \alpha\Phi(x_1) + \beta \Phi(x_2) = \alpha \Phi(z) + \beta \Phi(x).$$
Substituting for $\alpha,\beta$ and rearranging the equation above, we get by recasting $z-x$ as $(z-y)+(y-x)$, that 
\begin{equation}\label{eq: quotient_ineq}
\frac{\Phi(y)-\Phi(x)}{y-x}\leq \frac{\Phi(z)-\Phi(y)}{z-y}.
\end{equation}
This implies that the difference quotient for $\Phi$ is nondecreasing in $[c, d]$. Hence if $c < c_1 \leq x < y \leq d_1 < d$, we can see that \eqref{eq: quotient_ineq} can be extended to get
\begin{equation}\label{eq: quotient_ineq_2}
 \frac{\Phi(c_1)-\Phi(c)}{c_1-c}\leq \frac{\Phi(y)-\Phi(x)}{y-x}\leq \frac{\Phi(d)-\Phi(d_1)}{d-d_1}.
\end{equation}
It follows from this that
\begin{equation}
|\Phi(y)-\Phi(x)|\leq K|y-x|,
\end{equation}
for $K$ given by the maximum of the extreme terms in \eqref{eq: quotient_ineq_2} in absolute value. This implies that $\Phi$ satisfies a Lipschitz condition in $[c,d]$, and thus, in particular is absolutely continuous in $(a, b)$. We recall that $\Phi$ being absolutely continuous means that: for each $\varepsilon > 0$, there is a $\delta_{\varepsilon} > 0$ such that for any disjoint intervals $[a_i,b_i)\subset (a,b)$ satisfying $\sum_{i=1}^n|b_i-a_i|<\delta_{\varepsilon}$, then $\sum_{i=1}^n|\Phi(b_i)-\Phi(a_i)|<\varepsilon$. It follows then by the classical Lebesgue-Vitali theorem that
\begin{equation}\label{eq: integral_form}
\Phi(x)=\Phi(a)+\int_a^x\Phi'(t)dt,\hspace{8mm} a\leq x\leq b.
\end{equation}
We are now left with verifying the properties of $\Phi'$. Making $y = x + h$ in \eqref{eq: quotient_ineq}, with $h > 0$, we get
$$(D^+\Phi)(x) = \lim_{h\to 0^+}\frac{\Phi(x+h)-\Phi(x)}{h}\leq \frac{\Phi(d)-\Phi(d_1)}{d-d_1}<\infty,$$
and
$$(D^-\Phi)(x) = \lim_{h\to 0^+}\frac{\Phi(y)-\Phi(y-h)}{h}\geq \frac{\Phi(c_1)-\Phi(c)}{c_1-c} >-\infty.$$
Therefore the right and left derivatives of $\Phi$ exist at each point of $[c, d]$ and for $x < y$,
\begin{equation}
(D^-\Phi)(x) \leq \frac{\Phi(y)-\Phi(x)}{y-x}\leq (D^+\Phi)(x).
\end{equation}
Given that $(D^-\Phi)(x)\leq (D^+\Phi)(x)$ by \eqref{eq: quotient_ineq_2}, $(D^{\pm}\Phi)(\cdot)$ is increasing, and the set of discontinuities of these functions is, at most, a countable set. Hence $(D^-\Phi)(x)=(D^+\Phi)(x)$ at each continuity point of these functions, and this common value is $\Phi'$ of \eqref{eq: integral_form}. This establishes \eqref{eq: integral_theorem} with $\varphi=\Phi'$.\\
For the converse statement, assume \eqref{eq: integral_theorem} holds for $\Phi$. For $c <x < d$ consider the chord $L(x)$ joining $(c,\Phi(c))$ and $(d, \Phi(d))$ which is
given by
$$L(x)=\Phi(c)+\frac{\Phi(c)-\Phi(d)}{c-d}(x-c.)$$
To see that this is above the arc given by the graph of $\Phi$, i.e., $\geq \Phi(x)$, we have to show
\begin{equation}\label{eq: quotient_ineq3}
\frac{\Phi(x)-\Phi(c)}{x-c}\leq \frac{\Phi(d)-\Phi(c)}{d-c}, \hspace{8mm} c<x<d.
\end{equation}
Substituting \eqref{eq: integral_theorem} for $\Phi(\cdot)$, we verify it. It is not hard to see that
\begin{equation}
\frac{1}{x-c}\int_c^x\varphi(t)dt \leq \varphi(x) \leq \frac{1}{d-x}\int_x^d\varphi(t)dt
\end{equation}
since $\varphi(c) \leq \varphi(t)\leq \varphi(x) \leq \varphi(d)$, for $c < t < x < d$. The right side of \eqref{eq: quotient_ineq3} can be expressed as
\begin{equation*}
\begin{split}
\frac{\int_c^x\varphi(t)dt+\int_x^d\varphi(u)du}{(d-x)+(x-c)}&\geq\min \left(\frac{1}{x-c}\int_c^x\varphi(t)dt,\frac{1}{d-x}\int_x^d\varphi(u)du\right),\\
&=\frac{1}{x-c}\int_c^x\varphi(t)d.,
\end{split}
\end{equation*}
The last inequality is a consequence of the elementary relation
$$\min\left(\frac{a_1}{b_1},\frac{a_2}{b_2}\right)\leq \frac{a_1+a_2}{b_1+b_2}\leq\max\left(\frac{a_1}{b_1},\frac{a_2}{b_2}\right),$$
for any real numbers $a_1,a_2$ and positive numbers $b_1,b_2$. Thus \eqref{eq: quotient_ineq3} holds, and $\Phi$ given by \eqref{eq: integral_theorem} is convex in $(a, b)$ as claimed.
\end{proof}

Recall that a Young function $\Phi : \mathbb{R}\to \overline{\mathbb{R}}^+$ is convex and $\Phi(0) = 0$, but which may jump to $+\infty$ at a finite point, it is clear that if $\Phi(a) = +\infty$ for some $a > 0$, then $\Phi(x) = +\infty$ for all $x > a$. Interpreting $\varphi(0) = 0$ and $\varphi(x) = \infty$ for $x > a$ in this case, \eqref{eq: integral_theorem} continues to hold. Hence we may state the representation \eqref{eq: integral_theorem} for any Young function $\Phi$ in the following form.
\begin{cor}\label{cor: integral_form}
Let $\Phi : \mathbb{R}\to \overline{\mathbb{R}}^+$ be a young function. Then it can be represented as 
\begin{equation}\label{eq: integral_corollary}
\Phi(x)=\int_0^x\varphi(t)dt,\hspace{8mm} x\in \mathbb{R}^+,
\end{equation}
where $\varphi(0)=0, \varphi : \mathbb{R}\to \overline{\mathbb{R}}^+$ nondecreasing, left continuous and if $\varphi(x) = +\infty$ for $x > a$ then $\Phi(x) = +\infty$, for $x > a > 0$.
\end{cor}

We are now ready to give the proof of \eqref{eq: lower_ineq_Young}. Note that for any $a > 0$, we have by \eqref{eq: integral_corollary} and the mean value theorem for Lebesgue integrals, that
$$\frac{\Phi(a)}{a}=\frac{1}{a}\int_0^a\varphi(t)dt=\varphi(t^*), \hspace{5mm} \textup{ for some } 0<t^*<a. $$
Hence, by using Corollary \ref{cor: integral_form} for $\Phi^*$ that
\begin{equation}\label{eq: dual_ineq}
\begin{split}
\Phi^*\left(\frac{\Phi(a)}{a}\right)&=\int_0^{\Phi(a)/a}\varphi^*(t)dt=\frac{\Phi(a)}{a}\varphi^*(\tilde{t}), \hspace{5mm} 0<\tilde{t}<\frac{\Phi(a)}{a}=\varphi(t^*)\\
&<\frac{\Phi(a)}{a}\varphi^*(\varphi(a))\leq \frac{\Phi(a)}{a}a=\Phi(a),
\end{split}
\end{equation}
since $\varphi$ and $\varphi^*$ are inverse to each other. Note that here $\varphi^*$ represents the function in \eqref{eq: integral_corollary} corresponding to $\Phi^*$. Letting $\Phi(a) = \alpha$ we get from \eqref{eq: dual_ineq} that
$$\frac{\alpha}{\Phi^{-1}(\alpha)}<\Phi^*(\alpha), \hspace{3mm} \textup{ or } \hspace{3mm} \alpha<\Phi^{-1}(\alpha)(\Phi^*)^{-1}(\alpha).$$

\bibliographystyle{unsrt}
\bibliography{main}
\end{document}